\documentclass[11pt]{amsart}

\usepackage{graphicx}
\usepackage{mathrsfs}
\usepackage{color}
\usepackage{amssymb}
\usepackage{amsmath}
\usepackage{mathrsfs}
\usepackage[margin = 1.4in] {geometry}
\usepackage[hyperindex,breaklinks]{hyperref}
\usepackage{xcolor}

\newtheorem{prop}{Proposition}
\newtheorem{theo}[prop]{Theorem}
\newtheorem*{theo*}{Theorem}
\newtheorem{lemm}[prop]{Lemma}
\newtheorem{coro}[prop]{Corollary}

\newtheorem{rema}[prop]{Remark}
\newtheorem*{Ackn}{Acknowledgements}

\theoremstyle{definition}

\makeatletter

\newcommand{\Rmnum}[1]{\expandafter\@slowromancap\romannumeral #1@}
\makeatother

\newcommand{\RR}{\mathbf{R}}

\newcommand{\cA}{\mathcal A}

\newcommand{\bR}{\mathbb{R}}
\newcommand{\bS}{\mathbb{S}}

\DeclareMathOperator{\dis}{dist}

\DeclareMathOperator{\tr}{tr}

\DeclareMathOperator{\Vol}{Vol}

\DeclareMathOperator{\Ric}{Ric}
\DeclareMathOperator{\biRic}{biRic}

\DeclareMathOperator{\Div}{div}

\DeclareMathOperator{\secondfund}{II}
\DeclareMathOperator{\Rm}{Rm}

\newcommand{\bangle}[1]{\left\langle #1 \right\rangle}

\title[stable anisotropic minimal hypersurfaces]{Stable anisotropic minimal hypersurfaces \\ in $\bR^{5}$ and $\bR^{6}$}
\thanks{C.X. is supported by NSFC (Grant No. 12271449, 12126102) and the Natural Science Foundation of Fujian Province of China (Grant No. 2024J011008).}

\author{Jia Li}
\address{School of Mathematical Sciences\\
		Xiamen University\\
		361005, Xiamen, P.R. China}
\email{lijiamath@stu.xmu.edu.cn}
\author{Chao Xia}
\address{School of Mathematical Sciences\\
		Xiamen University\\
		361005, Xiamen, P.R. China}
\email{chaoxia@xmu.edu.cn}

\begin{document}

\begin{abstract}
In this paper, we prove that a complete, two-sided, stable anisotropic minimal immersed  hypersurface in $\bR^{5}$ or $\bR^{6}$ is flat, provided the anisotropic area functional is $C^4$-close to the area functional. 
\end{abstract}
\maketitle

\section{Introduction}
The classical Bernstein problem asks whether a minimal graph in $\bR^{n+1}$ is flat. The picture is nowadays complete that it has been answered affirmatively for $n\le 7$ by Bernstein \cite{Bernstein1927}, Fleming \cite{Fleming1962}, Almgren \cite{Almgren1966}, de Giorgi \cite{De Giorgi1965} and Simons \cite{Sim68} and negatively for $n\ge 8$ by Bombieri-de Giorgi-Giusti \cite{BDE}. The stable Bernstein problem asks whether a complete {two-sided} stable minimal hypersurface in $\bR^{n+1}$ is flat. It has been solved for $n=2$ by do Carmo-Peng and {Fischer Colbrie-Schoen} and Pogorelov, in \cite{do Carmo-Peng1979,FCS80,Pogorelov1981} respectively, and for $n\le 6$ by Schoen-Simon-Yau \cite{Schoen-Simon-Yau} under the Euclidean volume growth assumption (see also Bellettini \cite{Bel23}). 
Quite recently, there are important progress for the stable Bernstein problem. Precisely, it has been first solved by Chodosh-Li \cite{Chodosh-Li} for $n=3$, and two different proofs have been given by Chodosh-Li \cite{Chodosh-Li-anisotropic} and Catino-Mastrolia-Roncoroni \cite{Catino-Mastrolia-Roncoroni}. In particular,  Chodosh-Li's proof \cite{Chodosh-Li-anisotropic}, which is based on $\mu$-bubble method, yields the intrinsic Euclidean volume growth for a complete stable minimal hypersurface. This method has been successfully applied to solve stable Bernstein problem in $\bR^5$ by Chodosh-Li-Minter-Stryker \cite{Chodosh-Li-Minter-Stryker-5bernstein} and in $\bR^6$ by Mazet \cite{Mazet-6Bernstein}. In view of Bombieri-de Giorgi-Giusti's counterexample in $\bR^8$. The only remaining unsolved case is stable Bernstein problem in $\bR^7$.

For a two-sided immersion $M^n\to \bR^{n+1}$ with chosen unit normal vector field $\nu$, we consider the anisotropic area functional or parametric elliptic integral
$$\mathcal{A}_F(M):=\int_{M}F(\nu)d\mu,$$
where $F:\bR^{n+1}\setminus\{0\}\to(0,\infty)$ is a positive $C^3$ function which is $1$-homogeneous, i.e. $F(tz)=t F(z)$ for $t>0$,
and is elliptic which means $D^2F(z)$ is a positive definite endomorphism for all $z\in\mathbb{S}^n$. We say $M$ is $F$-minimal if $\frac{d}{dt}|_{t=0}\mathcal{A}_F(M_{t})=0$ for all compactly supported variations of $M$  and that $M$ is stable if in addition $\frac{d^2}{dt^2}|_{t=0}\mathcal{A}_F(M_{t})\geq 0$ for such variations. In particular, when $F(\nu)=|\nu|$, $\mathcal{A}_F$ reduces to the classical area functional, and the $F$-minimal and the stable $F$-minimal  hypersurfaces are classical minimal and stable minimal hypersurfaces, respectively. 
The $\mathcal{A}_F$ functionals have recently attracted significant attention
due to their practical applications and theoretical significance.

The (stable) anisotropic Bernstein problem asks the similar questions about flatness of $F$-minimal graphs or stable $F$-minimal hypersurfaces. The anisotropic Bernstein problem for graph cases has been resolved in the affirmative by Jenkins \cite{Jenkins1961} in $\bR^3$ and Simon \cite{Simon1977} in $\bR^4$. The stable anisotropic Bernstein problem has been resolved in the affirmative in $\bR^3$ assuming quadratic area growth by White \cite{White91}.
However, there are non-flat $\mathcal{A}_F$-minimizers in $\bR^{n+1}, n\geq 3$ by Mooney \cite{Mooney} and Mooney-Yang \cite{Mooney-Yang}. 
On the other hand, under the assumption of sufficient closeness of $\mathcal{A}_F$ to area, the stable anisotropic Bernstein problem has been resolved in the affirmative in $\bR^3$ by Lin \cite{F.H-Lin} (see also Colding-Minicozzi \cite{Colding-Minicozzi02}) and in $\bR^4$ by Chodosh-Li \cite{Chodosh-Li-anisotropic}, as well as in $\bR^{n+1}, 2\le n\le 5$ by Winklmann \cite{Winklmann} under the additional assumption of Euclidean volume growth following the seminal work of Schoen-Simon-Yau \cite{Schoen-Simon-Yau}. 
In particular, Chodosh-Li \cite{Chodosh-Li-anisotropic} used the $\mu$-bubble method to get the Euclidean volume growth for stable $F$-minimal hypersurfaces in $\bR^4$. We would like to mention that the $\mu$-bubble method has been initiated by Gromov \cite{Gromov2019Fourlecture} for problems involving scalar curvature and has many applications.
\begin{theo}[\cite{Chodosh-Li-anisotropic}]
Assume that $F$ satisfies
  \begin{align*}
 |\xi|^2\leq D^2F(z)(\xi,\xi)\leq \sqrt{2}|\xi|^2,
 \end{align*}
 for all $\xi\in z^{\perp}$. 
 Let $M^3\to\bR^{4}$ be a complete, two-sided, simply-connected, stable $F$-minimal immersion. Then there is some constant $C(F)$ such that $${\rm Vol}(B_R(p))\leq C(F)R^3,$$ where  $B_R(p)$ denotes the geodesic ball of radius $R$ centered at $p\in M$.\par
\end{theo}

Recall that the stable Bernstein problem in $\bR^{n+1}, n=4, 5$ has been solved by Chodosh-Li-Minter-Stryker \cite{Chodosh-Li-Minter-Stryker-5bernstein} and Mazet \cite{Mazet-6Bernstein}, again using Chodosh-Li's \cite{general soap bubbles} $\mu$-bubble method and some other new ingredients. It is natural to ask whether Chodosh-Li's result in the anisotropic case is true for $\bR^{n+1}, n=4,5$. This is the main aim for this paper. 
\begin{theo}\label{main-theo}
   Let $n=4,5$ and assume that $F$ satisfies
  \begin{align}\label{cond-1}
 |\xi|^2\leq D^2F(z)(\xi,\xi)\leq (1+\epsilon_{n})|\xi|^2,
 \end{align}
 for all $\xi\in z^{\perp}$, with $\epsilon_{4}=\frac{3}{20}$ and $\epsilon_{5}=\frac{1}{1000}$. 
 Let $M^n\to\bR^{n+1}$ be a complete, two-sided, stable $F$-minimal immersion.
  Then there is some constant $C(F)$ such that $${\rm Vol}(B_R(p))\leq C(F)R^n.$$
\end{theo}
Combining with Winklmann's result \cite{Winklmann}, we get
\begin{coro}
   Let $n=4,5$ and assume that $\cA_F$ is $C^4$-sufficiently close to the area. Then any complete,  two-sided,  stable $F$-minimal immersed hypersurface in $\bR^{n+1}$ is flat.
\end{coro}
Here, we say that  $\cA_{F}$ is $C^{4}$-sufficiently close to the area functional if $$\|F-1\|_{C^{4}(\bS^{n})}\leq \sigma_n$$ for some fixed $\sigma_n>0$. 

Let us first illustrate the $\mu$-bubble method to the Euclidean volume growth for stable minimal hypersurfaces developed by Chodosh-Li \cite{Chodosh-Li-anisotropic} in $\bR^4$, by Chodosh-Li-Minter-Stryker \cite{Chodosh-Li-Minter-Stryker-5bernstein} in $\bR^5$ and by Mazet \cite{Mazet-6Bernstein} in $\bR^6$. The first step is to consider a conformal metric $\tilde{g}=r^{-2}g$, where $r$ is the Euclidean distance to $p\in M$, inspired by Gulliver-Lawson \cite{GL}. Then the stability of the minimal hypersurface can be transferred to the property of a positive spectral lower bound for the scalar curvature ($n=3$), or the bi-Ricci curvature ($n=4$) or  some weighted bi-Ricci curvature ($n=5$), for the induced metric $\tilde{g}$ on $M\setminus\{p\}$. The next step is to find a good $\mu$-bubble that satisfies a  positive spectral lower bound for Ricci curvature. This leads to a spectral Bishop-Gromov volume estimate by Chodosh-Li-Minter-Stryker \cite{Chodosh-Li-Minter-Stryker-5bernstein} in three dimensions and Antonelli-Xu \cite{Antonelli-Xu} in any dimensions, by using isoperimetric profile method originally due to Bray \cite{Bray97}. The last step is to convert the volume estimate of $\mu$-bubble to the original metric $g$ and the growth of volume of $M$ is controlled by isoperimetric inequality of Michael-Simon \cite{Sobolev-ineq} and Brendle \cite{Brendle}.

As mentioned above, Chodosh-Li's strategy for $\bR^4$ works for 
 the anisotropic setting, when $F$ satisfies the pinching condition (\ref{cond-1}). We see that the above strategy also works for $\bR^5$ and $\bR^6$ when $D^2F$ is enough pinched. 
The main observation is that $H^2\le \delta_{n}^2 |A|^2$ for some small $\delta_{n}$ when $D^2F$ is enough pinched, though the classical mean curvature $H$ does not vanish in the anisotropic case. We first prove that $M$ has only one end, generalizing Cao-Shen-Zhu's result \cite{Cao-Shen-Zhu1997}. Then similar to the classical case, we show that the stability of $F$-minimal hypersurface can be transferred to the property of a positive spectral lower bound for the bi-Ricci curvature ($n=4$) or  some weighted bi-Ricci curvature ($n=5$), for the induced conformal metric $\tilde{g}$. The other part of proof follows the line as above.

\ 

\section{PRELIMINARIES}

We recall the first and the second variational formulas for the anisotropic area functional, see for example \cite[Section A]{Chodosh-Li-anisotropic}.
Let $M^{n}\to\bR^{n+1}$ be a two-sided immersed hypersurface and $M_t$ be a compact supported variation fixing $\partial M$ with a variational vector field $f\nu$, $f\in  C^{1}_{0}(M\setminus \partial M)$. Then 
$$\frac{d}{dt}\big|_{t=0}\mathcal{A}_F(M_{t})=\int_{M}\Div_{M}(DF(\nu))f.$$
The $F$-mean curvature is defined by $$H_{F}:=\Div_{M}(DF(\nu))=tr_M(\Psi(\nu)S_M),$$ 
where $S_{M}$ is the shape operator of $M$, $D$ is the connection in $\bR^{n+1}$, and the map $\Psi(\nu):T\bR^{n+1}\to T\bR^{n+1}$ is defined by $\Psi(\nu):V\mapsto D^{2}F(\nu)[V,\cdot]$.
Thus $M$ is $F$-minimal if and only if $H_F=0$.
Let $M^{n}\to\bR^{n+1}$ be an $F$-minimal two-sided immersed hypersurface, denoting $\nabla$ by the induced connection on $M$, then
  $$\frac{d^2}{dt^2}\big|_{t=0}\cA_{F}(M_{t})=\int_{M}\bangle{\nabla f,\Psi(\nu)\nabla f}-\tr_{M}(\Psi(\nu)S_{M}^2)f^2.$$
Thus, when $F$ satisfies the pinching condition \eqref{cond-1} and $M$ is stable $F$-minimal, then 
\begin{eqnarray}\label{delta-stable}\int_{M}|\nabla f|^2-\frac{1}{1+\epsilon_{n}}|A|^2f^2\geq 0,
\end{eqnarray}
 for any $f\in C^{1}_{0}(M\setminus \partial M)$, where $A$ is the second fundamental form of the immersion. 
\begin{rema}
    We mention that $M$ which satisfying inequality like \eqref{delta-stable} is called $\delta$-stable in the literature. For recent developments on  $\delta$-stable minimal hypersurfaces, we refer to \cite{HLW24} and the reference therein.
\end{rema}

The first variational formula leads to the following isoperimetric inequality.
 \begin{prop}[\cite{Chodosh-Li-anisotropic}]\label{iso-perimetric}
    Let $M^{n}\to\bR^{n+1}$ be a $F$-minimal immersion and the image of $\partial M$ is connected in $B^{\bR^{n+1}}_R$ for some $R>0$. Then $$|M|\leq\frac{R\|F\|_{C^{1}(\bS^{n})}}{n\cdot\min\limits_{\nu\in\bS^{n}}F(\nu)}|\partial M|,$$
    where $\|F\|_{C^{m}(\bS^n)}:=\big(\sum\limits_{j=0}^{m}\|D^{(j)}F\|^2_{C^{0}(\bS^n)}\big)^{\frac{1}{2}}$, and $B^{\bR^{n+1}}_R$ denotes the standard  ball of radius $R$ in $\bR^{n+1}$.
 \end{prop}
The stability of $M$ and the Michael-Simon-Sobolev inequality together lead to the $L^2$-Sobolev inequality on $M$, which implies that $M$ has infinite volume.
\begin{prop}[\cite{Chodosh-Li-anisotropic}]\label{infinite volume}
  Let $M^{n}\to\bR^{n+1}$ be a two-sided complete, stable $F$-minimal hypersurface, and $K$ is a compact subset of $M^{n}$. Then each unbounded component of $M\setminus K$ has infinite volume.
 \end{prop}
 On the other hand, combining the $L^2$-Sobolev inequality and Proposition \ref{infinite volume}, the same arguments as used by Cao-Shen-Zhu  \cite{Cao-Shen-Zhu1997} imply the following.
 \begin{prop}[\cite{Chodosh-Li-anisotropic}]\label{least-two ends}
   Let $M^{n}\to\bR^{n+1}$ be a complete two-sided, stable $F$-minimal immersed hypersurface with at least two ends, then there is a bounded non-constant harmonic function $u$ on $M$ with finite Dirichlet energy.
 \end{prop}
 
The following lemma establishes a control of the mean curvature $H$ by the second fundamental form $A$ of $M$, under the assumption (\ref{cond-1}) that $D^{2}F(\nu)$ is sufficiently pinched.
\begin{lemm}\label{mean-estimate}
  Assume $F$ satisfies (\ref{cond-1}). Let $M^{n}\to \bR^{n+1}$ be an $F$-minimal immersion. Then  \begin{eqnarray}\label{Hsmall}H^2\leq \delta_{n}^2|A|^2,\end{eqnarray}
  where the constant $\delta_{n}^2=\frac{(n-1)\epsilon_{n}^2}{(1+\epsilon_{n})^2}$.
\end{lemm}

\begin{proof}
 If we diagonalise the second fundamental form $A$ at a given point $p$, and denote $\kappa_{i}$ and $e_i$ for the principal curvatures and the corresponding principal directions, then the $F$-mean curvature at $p$ can be written as $$H_F=\sum\limits_{i=1}^{n}a_{i}\kappa_{i},$$
  where $a_{i}=D^2F(\nu)[e_i, e_i]$. Without loss of generality, we assume that $a_{1}\leq a_{2}\leq \cdots\leq a_{n}$. The pinching assumption (\ref{cond-1}) implies $$1\leq a_{1}\leq a_{2}\leq \cdots\leq a_{n}\leq 1+\epsilon_{n}.$$
  Since $H_F=0$, we have $$\kappa_{n}=-\frac{a_{1}\kappa_{1}+a_{2}\kappa_{2}+\cdots+a_{n-1}\kappa_{n-1}}{a_{n}}.$$ Write $\beta_{1}=\frac{a_{1}}{a_{n}}$, $\beta_{2}=\frac{a_{2}}{a_{n}}$,$\cdots$, $\beta_{n-1}=\frac{a_{n-1}}{a_{n}}$, with $\frac{1}{1+\epsilon_{n}}\leq\beta_{1}\leq\beta_{2}\leq\cdots\leq\beta_{n}\leq 1$. Then, we obtain
$$H^2=(\kappa_{1}+\kappa_{2}+\cdots+\kappa_{n})^2=((1-\beta_{1})\kappa_{1}+(1-\beta_{2})\kappa_{2}+\cdots+(1-\beta_{n-1})\kappa_{n-1})^2.$$
Using Cauchy-Schwarz's inequality, we get
\begin{align*}
  \frac{H^2}{|A|^2}&=\frac{((1-\beta_{1})\kappa_{1}+(1-\beta_{2})\kappa_{2}+\cdots+(1-\beta_{n-1})\kappa_{n-1})^2}{\kappa_{1}^2+\kappa_{2}^2+\cdots+\kappa_{n-1}^{2}+\kappa_{n}^2}\\
&\leq\frac{(n-1)((1-\beta_{1})^2\kappa_{1}^2+(1-\beta_{2})^2\kappa_{2}^2+\cdots+(1-\beta_{n-1})^2\kappa_{n-1}^2)}{\kappa_{1}^2+\kappa_{2}^2+\cdots+\kappa_{n-1}^2}\\
&\leq (n-1)(1-\frac{1}{1+\epsilon_{n}})^2=\delta_{n}^2.
\end{align*}
\end{proof}

\section{one-endness of stable $F$-minimal hypersurfaces in $\bR^{n+1}$}
In this section, we prove that any $n$-dimensional complete two-sided stable $F$-minimal immersed hypersurface $M^{n}\to \bR^{n+1}$ must have exactly one end, if $D^{2}F$ is sufficiently pinched.
\begin{lemm}\label{one-end-improve}
Assume $F$ satisfies (\ref{cond-1}). Let $M^n$ be a complete two-sided, stable $F$-minimal immersion in $\bR^{n+1}$ and $u$ be a harmonic function on $M$. Then
\begin{align*}
\Lambda_{n}\int_{M}\varphi^2|A|^2|\nabla u|^2+\frac{1}{n-1}\int_{M}\varphi^2|\nabla|\nabla u||^2\leq\int_{M}|\nabla \varphi|^2|\nabla u|^2.
\end{align*}
for any $\varphi\in C^{1}_{0}(M)$, where $\Lambda_{n}=\frac{1}{1+\epsilon_{n}}-\frac{n-1}{n}\left(1+\frac{n-2}{\sqrt{n-1}}\frac{\epsilon_{n}}{1+\epsilon_{n}}\right)$.
\end{lemm}
\begin{proof}
Fix $p\in M$, choose an orthonormal basis $\{e_{i}\}_{i=1}^{n}$ in $T_{p}M$. Denote $A_{ij}=A(e_{i},e_{j})$. By the Gauss equation,
\begin{align*}
  \Ric_{M}(e_{1},e_{1})
  &=A_{11}H-A_{11}^2-\sum\limits_{i=2}^{n}A_{1i}^2.
\end{align*}
Using this, we estimate $|A|^2$ as follows:
\begin{align*}
  |A|^2 &\geq A_{11}^2+\sum\limits_{i=2}^{n}A_{ii}^2+2\sum\limits_{i=2}^{n}A_{1i}^2 \\
   & \geq A_{11}^2+\frac{1}{n-1}(\sum\limits_{i=2}^{n}A_{ii})^2+2\sum\limits_{i=2}^{n}A_{1i}^2\\
  &=A_{11}^2+\frac{1}{n-1}(H^2-2HA_{11}+A_{11}^2)+2\sum\limits_{i=2}^{n}A_{1i}^2\\
&\geq -\frac{n}{n-1}\Ric_{M}(e_{1},e_{1})+\frac{n-2}{n-1}HA_{11}\\
  &\geq -\frac{n}{n-1}\Ric_{M}(e_{1},e_{1})- \frac{n-2}{n-1}\delta_{n}|A|^2,
  \end{align*}
where  we used the fact $|H|\geq-\delta_{n}|A|$ in the last inequality.
Thus, we obtain a lower bound on the  Ricci curvature:
\begin{align}\label{Ric-inq}
\Ric_{M}(e_{1},e_{1})\geq-\frac{n-1}{n}\left(1+ \frac{n-2}{n-1}\delta_{n}\right)|A|^2.
\end{align}
Since $M^n$ is stable $F$-minimal, the pinching condition (\ref{cond-1}) implies that
$$\int_{M}\frac{1}{1+\epsilon_{n}}|A|^2f^2\leq\int_{M}|\nabla f|^2, \forall f\in C^{1}_{0}(M).$$
Replacing $f$ by $|\nabla u|\varphi$, we rewrite the $F$-stability inequality as below:
\begin{align}\label{eq-x1}
  \frac{1}{1+\epsilon_{n}}\int_{M}\varphi^2|\nabla u|^2|A|^2 & \leq\int_{M}|\nabla\varphi|^2|\nabla u|^2+2\int_{M}(\varphi|\nabla u|\bangle{\nabla\varphi,\nabla|\nabla u|}+\varphi^2|\nabla|\nabla u||^2) \\
  &=\int_{M}|\nabla\varphi|^2|\nabla u|^2-\int_{M}\varphi^2|\nabla u|\Delta|\nabla u|.\nonumber
\end{align}
On the one hand, by the improved Kato's inequality for harmonic functions, $$|\nabla^{2}u|^2\geq\frac{n}{4(n-1)}|\nabla u|^{-2}|\nabla|\nabla u|^2|^{2}=\frac{n}{n-1}|\nabla|\nabla u||^2.$$
Combined with the Bochner formula
\begin{align*}
\Delta|\nabla u|^2&=2\Ric_{M}(\nabla u,\nabla u)+2|\nabla^2 u|^2,
\end{align*} the improved Kato inequality for harmonic functions
$$|\nabla^{2}u|^2\geq\frac{n}{4(n-1)}|\nabla u|^{-2}|\nabla|\nabla u|^2|^{2}=\frac{n}{n-1}|\nabla|\nabla u||^2,$$ and the inequality (\ref{Ric-inq}) with $e_1=\frac{\nabla u}{|\nabla u|}$, we have
$$\Delta|\nabla u|\geq-\frac{n-1}{n}\left(1+ \frac{n-2}{n-1}\delta_{n}\right)|A|^{2}|\nabla u|+\frac{1}{n-1}|\nabla u|^{-1}|\nabla|\nabla u||^2.$$
Inserting into \eqref{eq-x1}, we get the assertion.
\end{proof}
\begin{rema}\label{rema-2}
$\Lambda_{n}>0$ if $\epsilon_n>0$ is sufficiently small. In particular, in dimensions  $n=4,5$, we can choose the following pinching constant:
\begin{itemize}
  \item $n=4$, let $1+\epsilon_{4}=\frac{23}{20}$, then $\Lambda_{4}\approx 0.0335$. 
  \item $n=5$, let $1+\epsilon_{5}=\frac{1001}{1000}$, then $\Lambda_{5}\approx 0.1998$.
\end{itemize}
\end{rema}
By using a standard argument as Cao-Shen-Zhu \cite{Cao-Shen-Zhu1997}, we get the one-endness of $M$.
\begin{prop}\label{one-end}
  Assume $F$ satisfies (\ref{cond-1}) such that $\Lambda_n>0$, where $\Lambda_n$ is given in Lemma \ref{one-end-improve}. Then any complete, two-sided stable $F$-minimal immersion $M^{n}$ in $\bR^{n+1}$ has only one end.
\end{prop}
\begin{proof}
  We argue by contradiction. If not, then $M^{n}$ has at least two ends. Therefore, Proposition \ref{least-two ends} implies that $M^{n}$ admits a nontrivial harmonic function $u$ with finite Dirichlet energy, i.e. $\int_{M}|\nabla u|^2\leq C<\infty$. For $R>0$, take $\varphi\in C^{1}_{0}(M)$, such that $\varphi|_{B_{R}(0)}=1$, $\varphi|_{B_{2R}\setminus B_{R}}=0$ and $|\nabla \varphi|\leq\frac{2}{R}$. Then Lemma \ref{one-end-improve} implies that
  \[
  \int_{B_{R}}\Lambda_n|A|^2|\nabla u|^2+\frac{1}{n-1}|\nabla|\nabla u||^2\leq\frac{4}{R^2}\int_{M}|\nabla u|^2\leq\frac{4C}{R^2}.
  \]
  Sending $R\to\infty$, we conclude that $|\nabla|\nabla u||^2=0.$
  In particular, this implies that $|\nabla u|$ is a constant. But $u$ is not a constant, which means that $|\nabla u|>0$. Hence, this implies that $$\int_{M}1=\frac{1}{|\nabla u|^2}\int_{M}|\nabla u|^2<\infty,$$
  contradicts with Proposition \ref{infinite volume}.
\end{proof}

\section{A conformal deformation of metrics}
Let $M^n\to\bR^{n+1} (n=4,5)$ be a complete two-sided stable $F$-minimal immersed hypersurface passing through $0\in \bR^{n+1}$.
Let $g$ be the induced metric on $M$ and consider the conformal deformation by Gulliver-Lawson \cite{GL} on $M\setminus\{0\}$, i.e., $N=M\setminus\{0\}$ with metric $\tilde{g}=r^{-2}g$, where $r(x)=dist_{\bR^{n+1}}(0, x)$. We denote the related notations in $(N, \tilde{g})$ by a tilde.
We recall that the weighted bi-Ricci curvature $\biRic_{\alpha}$ is defined as:
$$\biRic_{\alpha}(e_{1},e_{2})=\sum\limits_{i=2}^{n}\Rm(e_{1},e_{i},e_{i},e_{1})+\alpha\sum\limits_{j=3}^{n}\Rm(e_{2},e_{j},e_{j},e_{2}),$$
and the minimum of $\biRic_{\alpha}$ is denoted by $$\lambda_{\biRic_{\alpha}}=\min_{v, w\in T_{p}M,~v\perp w,~|v|=|w|=1}\biRic_{\alpha}(v,w).$$
In the case $\alpha=1$, $\biRic_{\alpha}$ is the bi-Ricci curvature, which was first introduced by Shen-Ye in \cite{Shen-Ye96}.

The aim of this section is to prove the following
\begin{prop}\label{low-bound}
Let $n=4, 5$. Assume $F$ satisfies (\ref{cond-1}). Let $M^n\to\bR^{n+1}$ be a complete two-sided stable $F$-minimal immersed hypersurface passing through $0\in \bR^{n+1}$. Then for any $\varphi\in C^{1}_{0}(N,\tilde{g})$,
\begin{eqnarray}\label{ineq-x1}\int_{N}|\tilde{\nabla} \varphi|^2_{\tilde{g}}d\tilde{\mu}\geq \int_{N}(\tau_{n}-\eta_{n}\tilde{\lambda}_{\biRic_{\alpha}})\varphi^2d\tilde{\mu},\end{eqnarray}
  where $\tau_n>0$ and $\eta_n>0$ are given by
  \begin{itemize}
  \item $n=4$, $\quad \tau_4=\frac{7}{4}\cdot\frac{1+\epsilon_{4}-6\epsilon_{4}^2}{d_{4}(1+\epsilon_{4})^2}-\frac{9}{8}, \quad \eta_4=\frac{1+\epsilon_{4}-6\epsilon_{4}^2}{d_{4}(1+\epsilon_{4})^2},\quad d_{4}\approx 1.051.$
  \item $n=5$, $\quad \tau_5=\frac{5183}{1600}\cdot\frac{1+\epsilon_{5}-12\epsilon_{5}^2}{d_{5}(1+\epsilon_{5})^2}-\frac{39}{16}, \quad \eta_5=\frac{1+\epsilon_{5}-12\epsilon_{5}^2}{d_{5}(1+\epsilon_{5})^2},\quad d_{5}\approx 1.10536028.$
\end{itemize}
\end{prop}

Let $\overrightarrow{H}=-H\nu$, $\overrightarrow{x}$ denotes the position vector field in $\bR^{n+1}$. 
We first recall several formulas under the conformal deformation, which have been given in \cite{Chodosh-Li-Minter-Stryker-5bernstein, Mazet-6Bernstein}.
 \begin{prop}\label{conf-indentity}
   Let $(M^n,g)$ be a complete two-sided, stable $F$-minimal hypersurface in $\bR^{n+1}$, then the following equalities hold:\\
$(i)$. $$\tilde{\Delta}(\log r)=n-n|dr|^2+\small\bangle{\overrightarrow{H},\overrightarrow{x}},$$
$(ii)$.
$$r^{2}\Rm(e_{i},e_{j},e_{j},e_{i})={\tilde{\Rm}}(e_{i},e_{j},e_{j},e_{i})-2+|dr|^2+(dr(e_{i}))^2+(dr(e_{j}))^2+\bangle{\overrightarrow{x},\nu}(A_{ii}+A_{jj}).$$
 \end{prop}
{Using the Gauss equation and Proposition \ref{conf-indentity} (ii), we have expressions for the weighted bi-Ricci curvatures under the two different metrics. }
\begin{prop}\label{pro-2}
   \begin{align*}
   {\biRic}_{\alpha}(e_{1},e_{2})&=\sum\limits_{i=2}^{n}(A_{ii}A_{11}-A_{i1}^2)+\alpha\sum\limits_{j=3}^{n}(A_{22}A_{jj}-A_{2j}^2).\\
   \widetilde{\biRic}_{\alpha}(\tilde{e}_{1},\tilde{e}_{2})
 &=r^2(\sum\limits_{i=2}^{n}(A_{11}A_{ii}-A_{1i}^2)+\alpha\sum\limits_{j=3}^{n}(A_{22}A_{jj}-A_{2j}^2))\\
 &\quad+2((n-1)+\alpha(n-2))-(n+\alpha(n-1))|dr|^2 \\
&\quad-((n-2-\alpha)dr(e_{1})^2+\alpha(n-3)dr(e_{2})^2)\\
&-\bangle{\overrightarrow{x},\nu}((n-2-\alpha)A_{11}+\alpha(n-3)A_{22}+(1+\alpha)H).
 \end{align*}
 \end{prop}

We select the basis vectors $\tilde{e}_1$ and $\tilde{e}_2$ such that $\widetilde{\biRic_{\alpha}}(\tilde{e}_{1},\tilde{e}_2)=\tilde{\lambda}_{\biRic_{\alpha}}$.
Assume $\alpha\leq 1$ and $a\in \bR$ satisfy $a\geq\frac1 2$ and $2a\geq\alpha$. Let $\varsigma_{n}$ be a positive constant, then
 \begin{align*}
&ar^2|A|^2+r^2(\sum\limits_{i=2}^{n}(A_{11}A_{ii}-A_{1i}^2)+\alpha\sum\limits_{j=3}^{n}(A_{22}A_{jj}-A_{2j}^2))\\
 &\quad-\bangle{\overrightarrow{x},\nu}((n-2-\alpha)A_{11}+\alpha(n-3)A_{22}+(1+\alpha)H)\\
\geq & a r^2\sum\limits_{i=1}^{n}A_{ii}^2+r^2(\sum\limits_{i=2}^{n}A_{11}A_{ii}+\alpha\sum\limits_{j=3}^{n}A_{22}A_{jj})\\
 &\quad-\bangle{\overrightarrow{x},\nu}((n-2-\alpha)A_{11}+\alpha(n-3)A_{22}+(1+\alpha)H)\\
 \geq & a r^2\sum\limits_{i=1}^{n}A_{ii}^2+r^2(\sum\limits_{i=2}^{n}A_{11}A_{ii}+\alpha\sum\limits_{j=3}^{n}A_{22}A_{jj})\\
 &\quad-\bangle{\overrightarrow{x},\nu}((n-2-\alpha)A_{11}+\alpha(n-3)A_{22}+H)-\varsigma_{n}\alpha^{2}r^{2}H^{2}-\frac{1}{4\varsigma_{n}}\\
     :=&q^{T}Bq+C^Tq-\varsigma_{n}\alpha^{2}r^{2}H^{2}-\frac{1}{4\varsigma_{n}},
 \end{align*}
where 
 \begin{equation*}
 B=
 \begin{pmatrix}
   a & \frac{1}{2} &  \frac{1}{2} &  \cdots &  \frac{1}{2} \\
    \frac{1}{2} & a & \frac{\alpha}{2} & \cdots & \frac{\alpha}{2} \\
    \frac{1}{2} & \frac{\alpha}{2} & a & \cdots & 0 \\
   \vdots & \vdots & \vdots & \ddots & \vdots \\
    \frac{1}{2} & \frac{\alpha}{2} & 0 & \cdots & a\\
 \end{pmatrix},
 \quad q=r
 \begin{pmatrix}
   A_{11} \\
   A_{22} \\
   A_{33}\\
   \vdots \\
   A_{nn}\\
 \end{pmatrix},
\end{equation*}
and
 \begin{equation*}
   \quad C=-\bangle{\frac{\overrightarrow{x}}{r},\nu}
   \begin{pmatrix}
     n-1-\alpha \\
     \alpha(n-3)+1 \\
    1 \\
    \vdots \\
    1
   \end{pmatrix}=-\bangle{\frac{\overrightarrow{x}}{r},\nu}\bar{C}.
 \end{equation*}
 If $B$ is positive definite, by using the inequality
  \begin{align*}
 q^{T}Bq+C^{T}q\geq -\frac{1}{4}C^{T}B^{-1}C,
 \end{align*} we have
 \begin{align*}
&\varsigma_{n}\alpha^{2}r^{2}H^{2}+ar^2|A|^2+r^2(\sum\limits_{i=2}^{n}(A_{11}A_{ii}-A_{1i}^2)+\alpha\sum\limits_{j=3}^{n}(A_{22}A_{jj}-A_{2j}^2))\\
 &\quad-\bangle{\overrightarrow{x},\nu}((n-2-\alpha)A_{11}+\alpha(n-3)A_{22}+H)-\frac{1}{4\varsigma_{n}}\\
&\geq-b\bangle{\frac{\overrightarrow{x}}{r},\nu}^2-\frac{1}{4\varsigma_{n}},
 \end{align*}
 where  $b:=\frac{1}{4}\bar{C}^{T}B^{-1}\bar{C}$.
 Therefore, since we choose $\alpha\le 1$,  $H^{2}\leq \delta_{n}^{2}|A|^{2}$, let $d_{n}=a+\varsigma_{n}\alpha^{2}\delta_{n}^{2}$, then
 \begin{align*}
   d_{n}r^2|A|^2+\widetilde{\biRic}_{\alpha}(\tilde{e}_{1},\tilde{e}_{2}) &\geq-b\bangle{\frac{\overrightarrow{x}}{r},\nu}^2+2((n-1)+\alpha(n-2))-(n+\alpha(n-1))|dr|^2 \\
&\quad-((n-2-\alpha)dr(e_{1})^2+\alpha(n-3)dr(e_{2})^2)-\frac{1}{4\varsigma_{n}}\\
&\geq-b\bangle{\frac{\overrightarrow{x}}{r},\nu}^2+2((n-1)+\alpha(n-2))-(n+\alpha(n-1))|dr|^2 \\
&\quad-(n-2-\alpha)|dr|^2+(n-2-\alpha(n-2))dr(e_{2})^2-\frac{1}{4\varsigma_{n}}\\
&\geq-b(1-|dr|^2)+2(n-1+\alpha(n-2))-(2n-2+\alpha(n-2))|dr|^2-\frac{1}{4\varsigma_{n}}.
 \end{align*}
We reformulate the expression as follows:
\begin{prop}\label{pro-3}
  Let $n=4, 5$, and $a\geq\frac{1}{2}$, $2a\geq\alpha$, $\alpha\leq 1$, $\varsigma_{n}>0$. If $B$ is a positive definite matrix, then the following inequality holds:
  $$r^2|A|^2\geq-\frac{b}{d_{n}}(1-|dr|^2)+\frac{2}{d_{n}}(n-1+\alpha(n-2))-\frac{1}{d_{n}}(2n-2+\alpha(n-2))|dr|^2-\frac{1}{d_{n}}\tilde{\lambda}_{\biRic_\alpha}-\frac{1}{4\varsigma_{n}d_{n}}.$$
\end{prop}
We now adapt the estimate for $|A|^2$ from Proposition \ref{pro-3} to give a characterization of $F$-stability in terms of the $\alpha$-bi-Ricci curvature under the conformal deformation.

\begin{proof}[Proof of Proposition \ref{low-bound}]
On the one hand, under the conformal metric $\tilde{g}$, we have
  $$d\tilde{\mu}=r^{-n}d\mu \quad and \quad |\tilde{\nabla} f|^2_{\tilde{g}}=r^2|\nabla f|^2.$$
The $F$-stability inequality yields
   $$\int_{N}r^{n-2}|\tilde{\nabla} f|^2_{\tilde{g}}d\tilde{\mu}\geq\frac{1}{1+\epsilon_{n}}\int_{N}r^{n-2}(r^2|A|^2)f^2d\tilde{\mu},$$\
   for any $f\in C^{1}_{0}(N,\tilde{g})$. Take $f=r^{\frac{2-n}{2}}\varphi$ for $\varphi\in C^{1}_{0}(N,\tilde{g})$, then
   $$\tilde{\nabla} f=r^{\frac{2-n}{2}}\tilde{\nabla}\varphi-\frac{n-2}{2}r^{-\frac{n}{2}}\varphi\tilde{\nabla}r,$$
and
   $$|\tilde{\nabla f}|^2_{\tilde{g}}=r^{2-n}|\tilde{\nabla} \varphi|^2_{\tilde{g}}+\frac{(n-2)^2}{4}r^{-n}\varphi^2|\tilde{\nabla} r|^2_{\tilde{g}}-(n-2)r^{1-n}\varphi\bangle{\tilde{\nabla}\varphi,\tilde{\nabla}r}_{\tilde{g}}=:L_{1}+L_{2}+L_{3}.$$
   We obtain $$\int_{N}r^{n-2}L_{1}d\tilde{\mu}=\int_{N}|\tilde{\nabla}\varphi|^2_{\tilde{g}}d\tilde{\mu},$$
   and
   $$\int_{N}r^{n-2}L_{2}d\tilde{\mu}=\int_{N}\frac{(n-2)^2}{4}|dr|^2\varphi^2d\tilde{\mu},$$
    where we used the fact that $r^{-2}|\tilde{\nabla} r|^2_{\tilde{g}}=|dr|^2$.
   Integrating by parts, the third term yields
   \begin{align*}
     \int_{N}r^{n-2}L_{3}d\tilde{\mu}&=-\int_{N}\frac{n-2}{2}\bangle{\tilde{\nabla}(\varphi^2),\tilde{\nabla}(\log r)}_{\tilde{g}}d\tilde{\mu}\\
     &=\int_{N}\frac{n-2}{2}\tilde{\Delta}(\log r)\varphi^2d\tilde{\mu}\\
     &=\int_{N}(\frac{n(n-2)}{2}-\frac{n(n-2)}{2}|dr|^2+\frac{n-2}{2}\bangle{\overrightarrow{H},\overrightarrow{x}})\varphi^2d\tilde{\mu}.
    \end{align*}
    Put them together, by Lemma \ref{mean-estimate}, we obtain
    \begin{align*}
      \int_{N}|\tilde{\nabla} \varphi|^2_{\tilde{g}}d\tilde{\mu}&\geq \int_{N}\bigg(\frac{1}{1+\epsilon_{n}} r^2|A|^2-\frac{n(n-2)}{2}+\frac{(n-2)(n+2)}{4}|dr|^2-\frac{n-2}{2}\bangle{\overrightarrow{H},\overrightarrow{x}}\bigg)\varphi^2d\tilde{\mu}\\
      &\geq\int_{N}\bigg(\frac{1}{1+\epsilon_{n}} r^2|A|^2-\frac{n(n-2)}{2}+\frac{(n-2)(n+2)}{4}|dr|^2-(n-2)H^2r^2-\frac{n-2}{16}\bigg)\varphi^2d\tilde{\mu}\\
      &\geq\int_{N}\bigg(\frac{(1+\epsilon_{n})-(n-2)(n-1)\epsilon_{n}^2}{(1+\epsilon_{n})^2}r^2|A|^2+\frac{(n-2)(n+2)}{4}|dr|^2-\frac{(8n+1)(n-2)}{16}\bigg)\varphi^2d\tilde{\mu}.
    \end{align*}

 On the other hand, combining with Proposition \ref{pro-3}, we consider the following two cases:
 
\textbf{Case 1.} For $n=4$, $1+\epsilon_{4}=\frac{23}{20}$, choosing $a=\alpha=1$, $\varsigma_{4}=1$, by software Mathematica, we can easily check $B$ is positive definite and $b=\frac{3}{2}$, $d_{4}\approx1.051$, then the above inequality yields:
\begin{align*}
&\frac{1+\epsilon_{4}-6\epsilon_{4}^2}{(1+\epsilon_{4})^2}r^2|A|^2-\frac{33}{8}+3|dr|^2\\
&\geq \frac{1+\epsilon_{4}-6\epsilon_{4}^2}{d_{4}(1+\epsilon_{4})^2}\bigg(\frac{33}{4}-\frac{13}{2}|dr|^2-\tilde{\lambda}_{\biRic_{\alpha}}\bigg)-\frac{33}{8}+3|dr|^2\\
&= \frac{33}{4}\cdot\frac{1+\epsilon_{4}-6\epsilon_{4}^2}{d_{4}(1+\epsilon_{4})^2}-\frac{33}{8}+\big(3-\frac{13}{2}\cdot\frac{1+\epsilon_{4}-6\epsilon_{4}^2}{d_{4}(1+\epsilon_{4})^2}\big)|dr|^2-\frac{1+\epsilon_{4}-6\epsilon_{4}^2}{d_{4}(1+\epsilon_{4})^2}\tilde{\lambda}_{\biRic_\alpha}.\\
 &\geq \frac{7}{4}\cdot\frac{1+\epsilon_{4}-6\epsilon_{4}^2}{d_{4}(1+\epsilon_{4})^2}-\frac{9}{8}-\frac{1+\epsilon_{4}-6\epsilon_{4}^2}{d_{4}(1+\epsilon_{4})^2}\tilde{\lambda}_{\biRic_\alpha}\\
 &:=\tau_{4}-\eta_{4}\tilde{\lambda}_{\biRic_\alpha},
\end{align*}
where $\tau_{4}\approx 0.1529$, $\eta_{4}\approx 0.7303$. In the second inequality, we used the fact that $$3-6.5\cdot\frac{1+\epsilon_{4}-6\epsilon_{4}^2}{d_{4}(1+\epsilon_{4})^2}\approx-1.7466<0$$ and
$|dr|\le 1$. 

\textbf{Case 2.} For $n=5$, $1+\epsilon_{5}=\frac{1001}{1000}$, 
if we choose $a=\frac{217}{200}$, $\alpha=\frac{11}{12}$, $\varsigma_{5}=160$, then $b=5.0311$ and $B$ is positive definite, $d_{5}\approx 1.0855$, the above inequality leads to:
\begin{align*}
&\frac{1+\epsilon_{5}-12\epsilon_{5}^2}{(1+\epsilon_{5})^2}r^2|A|^2-\frac{123}{16}+\frac{21}{4}|dr|^2\\
&\geq\frac{1+\epsilon_{5}-12\epsilon_{5}^2}{d_{5}(1+\epsilon_{5})^2}(-b+8+6\alpha-\frac{1}{64})-\frac{123}{16}+\bigg(\frac{1+\epsilon_{5}-12\epsilon_{5}^2}{d_{5}(1+\epsilon_{5})^2}(b-(8+3\alpha))+\frac{21}{4}\bigg)|dr|^2\\
 &\quad-\frac{1+\epsilon_{5}-12\epsilon_{5}^2}{d_{5}(1+\epsilon_{5})^2}\tilde{\lambda}_{\biRic_{\alpha}}\\
&= (\frac{651}{200}-\frac{1}{64})\cdot\frac{1+\epsilon_{5}-12\epsilon_{5}^2}{d_{5}(1+\epsilon_{5})^2}-\frac{39}{16}-\frac{1+\epsilon_{5}-12\epsilon_{5}^2}{d_{5}(1+\epsilon_{5})^2}\tilde{\lambda}_{\biRic_{\alpha}}\\
&\geq \frac{5183}{1600}\cdot\frac{1+\epsilon_{5}-12\epsilon_{5}^2}{d_{5}(1+\epsilon_{5})^2}-\frac{39}{16}-\frac{1+\epsilon_{5}-12\epsilon_{5}^2}{d_{5}(1+\epsilon_{5})^2}\tilde{\lambda}_{\biRic_{\alpha}}\\
 &:= \tau_{5}-\eta_{5}\widetilde{\lambda}_{\biRic_{\alpha}},
 \end{align*}
 where $\tau_{5}\approx 0.09181$, $\eta_{5}\approx 0.92027$.  In the second inequality, we used the fact that $$\frac{1+\epsilon_{5}-12\epsilon_{5}^2}{d_{5}(1+\epsilon_{5})^2}(b-(8+3\alpha))+\frac{21}{4}\approx-0.012942371<0$$ and
$|dr|\le 1$. 
\end{proof}

\section{$\mu$-bubbles and proof of Theorem 2}
In this section, under the condition of spectral uniformly positivity of the $\alpha$-bi-Ricci curvature on $(N,\tilde{g})$, we construct a warped $\mu$-bubble whose volume has an upper bound. We follow closely Chodosh-Li-Minter-Stryker \cite{Chodosh-Li-Minter-Stryker-5bernstein} and Mazet \cite{Mazet-6Bernstein}. 

First, by \cite[Theorem 1]{FCS80}, there exists a positive function $w$ on $(N,\tilde{g})$ satisfying
\begin{align}\label{JACfun}
-\tilde{\Delta}w=(\tau_{n}-\eta_{n}\tilde{\lambda}_{\biRic_{\alpha}})w.
\end{align}
We use $w$ for the weight function in $\cA_{k}(\Omega)$ as below.

\subsection{Warped $\mu$-bubbles}
Given an $n$-dimensional compact Riemannian manifold $(M_{0},\tilde{g})$ with boundary consisting of two disjoint components, denoted by $\partial M_{0}=\partial_{+}M_{0}\cup\partial_{-}M_{0}$, where neither of them are empty. Let $h$ be a smooth function on $\mathring{M_{0}}$(interior of $M_0$) with $h\to \pm\infty$ on $\partial_{\pm}M_{0}$, choose a finite perimeter set $\Omega_{0}$ with $\partial\Omega_{0}\subset\mathring{M_{0}}$ and $\partial_{+}M_{0}\subset\Omega_{0}$. We consider the following warped prescribed mean curvature functional $\cA_{k}$ defined by $$\cA_{k}({\Omega})=\int_{\partial^{*} \Omega}w^{k}d\tilde{\mu}^{n-1}-\int_{M_{0}}(\chi_{\Omega}-\chi_{\Omega_{0}})hw^{k}d\tilde{\mu}^{n}.$$

If there exists a finite perimeter set $\Omega$ minimizing $\cA_{k}$ among the Caccioppoli sets containing $\partial_{+}M_{0}$, we call $\Sigma=\partial^{*}\Omega$ a warped $\mu$-bubble, where $\partial^{*}\Omega$ denotes the reduced boundary of $\Omega$.  Let $\tilde{\nu}$ be the outward unit normal vector field of $\Sigma$. We note that the $\mu$-bubble was first introduced by Gromov \cite{Gromov2019Fourlecture}, whose existence and regularity theory have been established in  \cite{general soap bubbles,zhujintian2021}.
The first and second variation formulas are as follows. We refer to \cite{Hong-cmc nonexis,Mazet-6Bernstein} for details.
\begin{prop}[first and second variation formula]
  Let $\Omega_{t}$ be a smooth family of sets of finite perimeter with $\Omega_{t}\big|_{t=0}=\Omega$ and with velocity $\psi\tilde{\nu}$ at $t=0$, then $$\frac{d}{dt}\big|_{t=0}\cA_{k}(\Omega_{t})=\int_{\Sigma}(\tilde{H}_\Sigma w^{k}+\bangle{\tilde{\nabla}_{M_0} w^{k},\tilde{\nu}}-hw^{k})\psi d\tilde{\mu}^{n-1},$$
  where $\tilde{H}_\Sigma$ is the mean curvature of $\Sigma$. In particular, the $\mu$-bubble $\Sigma$ satisfies $$\tilde{H}_\Sigma=-kw^{-1}\bangle{\tilde{\nabla}_{M_0}w,\tilde{\nu}}+h.$$Also,
\begin{align*}
  \frac{d^2}{dt^2}\bigg|_{t=0}\cA_{k}(\Omega)
&=\int_{\Sigma}w^{k}\bigg(|\tilde{\nabla}_{\Sigma}\psi|^2-(|\tilde{A}_{\Sigma}|^2+\tilde{\Ric}_{M_{0}}(\tilde{\nu},\tilde{\nu}))\psi^2-kw^{-2}\bangle{\tilde{\nabla}_{M_{0}} w,\tilde{\nu}}^2\psi^2\\
  &\quad+kw^{-1}(\tilde{\Delta}_{M_{0}}w-\tilde{\Delta}_{\Sigma}w-\tilde{H}_\Sigma\bangle{\tilde{\nabla}_{M_0} w,\tilde{\nu}})\psi^2+\bangle{\tilde{\nabla}_{M_0} h,\tilde{\nu}}\psi^2\bigg)d\tilde{\mu}^{n-1}.
  \end{align*}
\end{prop}

\subsection{Volume estimates of $\mu$-bubbles}
For manifolds satisfying a spectral Ricci curvature lower bound, Chodosh-Li-Minter-Stryker \cite{Chodosh-Li-Minter-Stryker-5bernstein} and Antonelli-Xu in \cite{Antonelli-Xu} established the following spectral Bishop-Gromov volume comparison theorem.
\begin{theo}[\cite{Chodosh-Li-Minter-Stryker-5bernstein, Antonelli-Xu}]\label{Antonelli-Xu}
 Let $\Sigma^n$, $n\geq 3$, be an $n$-dimensional simply connected compact smooth manifold, and let $0\leq \theta\leq \frac{n-1}{n-2}$, $\gamma>0$. Assume there is a positive function $\omega\in C^{\infty}(\Sigma)$ such that:
  $$\theta\Delta_{\Sigma}\omega\leq \lambda_{\Ric} \omega-(n-1)\gamma \omega,$$
where $\lambda_{\Ric}$ is the smallest eigenvalue of the  Ricci tensor on $\Sigma$, then there exists a sharp volume bound $$\Vol(\Sigma^{n})\leq\gamma^{-\frac{n}{2}}\Vol(\bS^{n}).$$
\end{theo}

In the following, we find good $\mu$-bubble in $(N, \tilde{g})$ and 
apply the spectral Bishop-Gromov volume comparison theorem to get the volume estimate for the $\mu$-bubble.
\begin{prop}\label{diam-volume bound}
  Let $n=4, 5$ and $(N_{0},\tilde{g})$ be a compact subset of $(N^{n},\tilde{g})$ with connected boundary. Suppose there exists a point $p\in N_{0}$ such that $d_{\tilde{g}}(p,\partial N_{0})\geq 10\pi$. Let $w$ be a positive function on $(N_{0},\tilde{g})$ such that
\begin{align}\label{con-10}
  -\tilde{\Delta}_N w\geq(\tau_{n}-\eta_{n}\tilde{\lambda}_{\biRic_{\alpha}})w.
\end{align}
  Then there exists a connected relative open set $\Omega$ containing $\partial N_{0}$ such that $\Omega\subset B^{\tilde{g}}_{10\pi}(\partial N_{0})$ and 
 $$\Vol_{\tilde{g}}(\Sigma)\leq\left(\frac{\tau_{n}}{2(n-2)\alpha_{n}\eta_{n}}\right)^{-\frac{n-1}{2}}\Vol(\bS^{n-1}),$$
  where $\Sigma$ is a connected component of $\partial\Omega\setminus\partial N_{0}$, $B^{\tilde{g}}_{10\pi}(\partial N_0)$ denotes the $10\pi$-neighborhood of $\partial N_0$ with respect to the metric $\tilde{g}$. Here $\alpha_{4}=1$, $\alpha_{5}=\frac{11}{12}$.
\end{prop}

\begin{proof}
According to the second variation formula, following the same computations from \cite[Section 4.2]{Mazet-6Bernstein}, and using the spectral $\alpha$-bi-Ricci condition (\ref{con-10}), we obtain
\begin{align*}
  \frac{4}{4-k}\int_{\Sigma}|\tilde{\nabla}_{\Sigma}\psi|^2 
  &\geq\int_{\Sigma}\psi^2\bigg(k\tau_{n}-\alpha_{n}\tilde{\lambda}_{\Ric}+|\tilde{A}_{\Sigma}|^2+\alpha_{n} \tilde{H}_{\Sigma}\tilde{A}_{11}^{\Sigma}-\alpha_{n}\sum\limits_{i=1}^{n-1}(\tilde{A}_{1i}^{\Sigma})^2+k(\tilde{d}\ln w(\tilde{\nu}))^2\\
  &\quad+ k\tilde{H}_{\Sigma}\tilde{d}\ln w(\tilde{\nu})+\tilde{d}h(\tilde{\nu})\bigg).
\end{align*}
Choose $k=\frac{1}{\eta_{n}}$, and set
\begin{align*}
L_{n}&=
|\tilde{A}_{\Sigma}|^2+\alpha_{n} \tilde{H}_{\Sigma}\tilde{A}_{11}^{\Sigma}-\alpha_{n}\sum\limits_{i=1}^{n-1}(\tilde{A}_{1i}^{\Sigma})^2+\frac{1}{k}(\tilde{H}_{\Sigma}-h)^2+\tilde{H}_{\Sigma}(h-\tilde{H}_{\Sigma}).
\end{align*}
If we can show that 
\begin{align}\label{Ln-estimate}
L_{n}>\beta_{n} h^2,
\end{align}
then
$$\frac{4}{4-k}\int_{\Sigma}|\tilde{\nabla}_{\Sigma}\psi|^2\geq\int_{\Sigma}\psi^2\left(\frac{\tau_{n}}{2\eta_{n}}-\alpha_{n}\tilde{\lambda}_{\Ric}\right)+\psi^2\left(\frac{\tau_{n}}{2\eta_{n}}+\beta_{n} h^2+\tilde{d}h(\tilde{\nu})\right).$$
Following the method of \cite{general soap bubbles,Chodosh-Li-Minter-Stryker-5bernstein,Hong-cmc nonexis,Mazet-6Bernstein}, we can construct a smooth function $h$ such that $\frac{\tau_{n}}{2\eta_{n}}+\beta_{n} h^2-|\tilde{\nabla}_{N}h|>0$ holds on $\Sigma$.
Therefore, there exists a positive Jacobi function $\omega$, such that $$\frac{4\eta_{n}}{(4\eta_{n}-1)\alpha_{n}}\tilde{\Delta}_{\Sigma}\omega\leq\left(\tilde{\lambda}_{\Ric}-\frac{\tau_{n}}{2\alpha_{n}\eta_{n}}\right)\omega.$$
To prove inequality (\ref{Ln-estimate}) holds, we argue as in \cite[Section 4.2]{Mazet-6Bernstein}, write $L_{n}$ as a quadratic form and denoted $S_{n}$ by the associated matrix. Then, it suffices to show
\begin{equation*}
  S_{n}=
  \begin{pmatrix}
    \frac{1}{n-1}+\frac{\alpha_{n}}{n-1}-\frac{\alpha_{n}}{(n-1)^2}+\frac{1}{k}-1 & \frac{\alpha_{n}}{2}\sqrt{\frac{n-2}{n-1}}(1-\frac{2}{n-1}) & \frac{1}{2}-\frac{1}{k} \\
    \frac{\alpha_{n}}{2}\sqrt{\frac{n-2}{n-1}}(1-\frac{2}{n-1}) & 1-\frac{n-2}{n-1}\alpha_{n} & 0\\
    \frac{1}{2}-\frac{1}{k} & 0 & \frac{1}{k}-\beta_{n}\\
  \end{pmatrix}
  >0.
\end{equation*}
We again divide the discussion into two cases.

\textbf{Case 1.}
For $n=4$, the matrix $S_{4}$ becomes
\begin{equation*}
  S_{4}=
  \begin{pmatrix}
    \frac{2\alpha_{4}}{9}+\frac{1}{k}-\frac{2}{3} & \frac{\alpha_{4}}{6}\sqrt{\frac{2}{3}} & \frac{1}{2}-\frac{1}{k} \\
    \frac{\alpha_{4}}{6}\sqrt{\frac{2}{3}} & 1-\frac{2}{3}\alpha_{4} & 0\\
    \frac{1}{2}-\frac{1}{k} & 0 & \frac{1}{k}-\beta_{4}\\
  \end{pmatrix}.
\end{equation*}
According to Proposition \ref{low-bound}, $\alpha_{4}=1$, $\tau_{4}\approx 0.1529$, let $\frac{1}{k}=\eta_{4}\approx0.7303$. If we choose $\beta_{4}=\frac{2}{5}$, using software Mathematica, we find $S_{4}$ is positive definite, then $L_{4}>\frac{2}{5}h^2$.

 Since  $$\frac{4\eta_{4}}{(4\eta_{4}-1)\alpha_{4}}\approx1.5205\leq\frac{n-2}{n-3}=2,$$
we can use Theorem \ref{Antonelli-Xu} to conclude  that 
$$\Vol_{\tilde{g}}(\Sigma^3)\leq (\frac{\tau_{4}}{4\alpha_{4}\eta_{4}})^{-\frac{3}{2}}\Vol(\bS^3)\leq 168\pi^2.$$

\textbf{Case 2.} For $n=5$, the matrix $S_{5}$ yields
\begin{equation*}
  S_{5}=
  \begin{pmatrix}
    \frac{3\alpha_{5}}{16}+\frac{1}{k}-\frac{3}{4} & \frac{\alpha_{5}}{4}\sqrt{\frac{3}{4}} & \frac{1}{2}-\frac{1}{k} \\
    \frac{\alpha_{5}}{4}\sqrt{\frac{3}{4}} & 1-\frac{3}{4}\alpha_{5} & 0\\
    \frac{1}{2}-\frac{1}{k} & 0 & \frac{1}{k}-\beta_{5}\\
  \end{pmatrix}.
\end{equation*}
Note that $\alpha_{5}=\frac{11}{12}$, $\tau_{5}\approx 0.09181$, choose $\frac{1}{k}=\eta_{5}\approx 0.92027$, and let $\beta_{5}=\frac{1}{10}$, we conclude that $S_{5}$ is positive definite, then $L_{5}>\frac{1}{10}h^2$.

Similarly, since $$\frac{4\eta_{5}}{(4\eta_{5}-1)\alpha_{5}}\approx 1.497799542\leq \frac{n-2}{n-3}=\frac{3}{2},$$ we can use Theorem \ref{Antonelli-Xu} to conclude  that $$\Vol_{\tilde{g}}({\Sigma}^4)\leq(\frac{\tau_{5}}{6\alpha_{5}\eta_{5}})^{-2}\Vol(\bS^{4})\leq 8200\pi^2.$$
\end{proof}

\subsection{Volume growth of the geodesic ball}
In this subsection, we will complete the proof of Theorem \ref{main-theo}. Since all preceding results are obtained under the conformal metric $\tilde{g}$, we now estimate the volume of $B_{R}(p)$. 

\begin{proof}[Proof of Theorem \ref{main-theo}]
We first assume that $M$ is simply connected.
Up to a translation, we shall identify $p$ with the origin $0$, and let $\bar{r}(x)=\dis_{M,g}(x,0)$. For a fixed sufficiently large $R>0$, we consider the geodesic ball $B_{e^{10\pi}R}(0)$ with respect to the metric $g$. By Proposition \ref{one-end}, $M\setminus B_{e^{10\pi}R}(0)$ must have only one unbounded connected component $E$. We set $M^{\prime}=M\setminus E$. As argued in the proof of \cite[Theorem 1.4]{Chodosh-Li-anisotropic}, by using simply-connectedness of $M$, we see $\partial M^{\prime}=\partial E$ is connected.

Applying Proposition \ref{diam-volume bound} to $(M^{\prime}\setminus\{0\},\tilde{g})$, there exists a connected open set $\Omega$ in the $10\pi$-neighborhood of $\partial M^{\prime}$, such that each connected component of $\partial \Omega\setminus\partial M^{\prime}$ has volume bounded by $(\frac{\tau_{n}}{2(n-2)\alpha_{n}\eta_{n}})^{-\frac{n-1}{2}}\Vol(\bS^{n-1})$.

Let $M_{1}$ be the connected component of $M^{\prime}\setminus\Omega$ that contains $0$. Analogously, we can apply the same argument as in \cite{Chodosh-Li-anisotropic} to show that $M\setminus M_{1}$ is connected, and $M_{1}$ has only one boundary component denoted by $\Sigma=\partial M_{1}$. By the second item in \cite[Lemma 6.2]{Chodosh-Li-anisotropic}, we have $\min\limits_{x\in\Sigma}\bar{r}(x)\geq R$, so $B_{R}(0)\subset M_{1}$. On the other hand, by comparing the intrinsic distance $\bar{r}$ with the extrinsic distance $r$, we have $\max\limits_{x\in\Sigma}r(x)\leq e^{10\pi} R$, so 
\begin{align*}\Vol_{g}(\Sigma)&=\int_{\Sigma}r^{n-1} d\tilde{\mu}^{n-1}\leq e^{10(n-1)\pi}R^{n-1}\Vol_{\tilde{g}}(\Sigma)\\&\leq e^{10(n-1)\pi}\left(\frac{\tau_{n}}{2(n-2)\alpha_{n}\eta_{n}}\right)^{-\frac{n-1}{2}}\Vol(\bS^{n-1})R^{n-1}.\end{align*}
Finally, Proposition \ref{iso-perimetric} implies that \begin{align*}\Vol_{g}(B_{R}(0))&\leq  \Vol_{g}(M_{1})\leq\frac{\|F\|_{C^{1}(\bS^{n})}}{n\min\limits_{\nu\in\bS^{n}}F(\nu)}e^{10\pi}R \Vol_{g}(\Sigma)\\&\leq \frac{e^{10n\pi}\left(\frac{\tau_{n}}{2(n-2)\alpha_{n}\eta_{n}}\right)^{-\frac{n-1}{2}}\Vol(\bS^{n-1})\|F\|_{C^{1}(\bS^{n})}}{n\min\limits_{\nu\in\bS^{n}}F(\nu)}R^n.\end{align*}
\begin{itemize}
    \item $n=4$,\quad\quad $C(F)=\frac{42\pi^2e^{40\pi}\|F\|_{C^{1}(\bS^{4})}}{\min\limits_{\nu\in\bS^{4}}F(\nu)}.$ 
    \item  $n=5$, \quad\quad$C(F)=\frac{1640\pi^2e^{50\pi}\|F\|_{C^{1}(\bS^{5})}}{\min\limits_{\nu\in\bS^{5}}F(\nu)}$. 
\end{itemize}
 This completes the proof in the case that $M$ is simply connected.

When $M$ is not simply connected, we lift it to its universal cover $\tilde{M}\subset \mathbb{R}^{n+1}$. We remark that in general $F$-stability may not hold on $\tilde{M}$. Nevertheless, by the local isometry of $M$ and $\tilde{M}$, we see that the local property \eqref{Hsmall} holds on $\tilde{M}$. Also by the argument of Fischer Colbrie-Schoen \cite{FCS80}, the modified stability inequality \eqref{delta-stable} holds on $\tilde{M}$. By using \eqref{Hsmall} and \eqref{delta-stable}, we can use the same argument to show that $\tilde{M}$ has only one end and it holds \eqref{ineq-x1} on $\tilde{M}$. This enables to apply Proposition \ref{diam-volume bound} and the rest of the proof is the same. The proof is completed.
\end{proof}

\

\begin{Ackn}
{\rm The authors would like to express their deep gratitude to the referee for his/her carefully reading and valuable suggestions which improve the context of this paper.}
\end{Ackn}

\

\bibliographystyle{plain}

\end{document}